\documentclass[11pt]{amsart}
\usepackage[margin=1.35in]{geometry} 
\usepackage{amsmath,amsthm,amssymb,amsfonts,mathtools}
\usepackage[utf8]{inputenc}
\usepackage[cyr]{aeguill}
\usepackage{graphicx}
\usepackage{nicefrac}
\usepackage{stmaryrd}
\usepackage{amsthm}
\usepackage{mathrsfs} 
\usepackage{faktor}

\usepackage{dsfont}
\usepackage{tikz}
\usepackage{xcolor}
\usepackage{yfonts}
\usepackage{colonequals}
\usepackage{mathabx}
\usepackage{tikz-cd}
\usepackage{hyperref}
\definecolor{darkgoldenrod}{rgb}{0.12, 0.28, 0.59}
\hypersetup{colorlinks=true, citecolor=darkgoldenrod, linkcolor=darkgoldenrod}
\usepackage[foot]{amsaddr}
\usepackage[french, english]{babel}
\usepackage{comment}
\usepackage{orcidlink}

\DeclareFontFamily{U}{wncy}{}
\DeclareFontShape{U}{wncy}{m}{n}{<->wncyr10}{}
\DeclareSymbolFont{mcy}{U}{wncy}{m}{n}
\DeclareMathSymbol{\Sh}{\mathord}{mcy}{"58}

\newcommand{\gq}{{\mathfrak{q}}}

\newcommand{\gp}{{\mathfrak{p}}}

\newcommand{\Gal}{{\mathrm{Gal}}}

\newcommand{\Sel}{{\mathrm{Sel}}}

\newcommand{\Hom}{\mathrm{Hom}}

\newcommand{\ord}{\mathrm{ord}}
\newcommand{\Q}{{\mathbf Q}}
\newcommand{\Z}{{\mathbf Z}}

\newcommand{\R}{{\mathbf R}}

\newcommand{\QQ}{\mathbf{Q}}
\newcommand{\ZZ}{\mathbf{Z}}
\newcommand{\Qp}{\mathbf{Q}_p}
\newcommand{\Zp}{\mathbf{Z}_p}

\definecolor{Green}{rgb}{0.0, 0.5, 0.0}

\theoremstyle{definition}

\newtheorem{lemma}{Lemma}[section]
\newtheorem{definition}[lemma]{Definition}
\newtheorem{remark}[lemma]{Remark}
\newtheorem{corollary}[lemma]{Corollary}
\newtheorem{proposition}[lemma]{Proposition}
\newtheorem{conjecture}[lemma]{Conjecture}
\newtheorem{theorem}[lemma]{Theorem}

\title[]{Refined conjectures on Fitting ideals of Selmer groups over $\mathbf{Z}_p^2$-extensions}
\author[Cédric Dion]{Cédric Dion\,\orcidlink{https://orcid.org/0000-0002-5593-7824}}
\address[C.~Dion]{D\'epartement de Math\'ematiques et de Statistique, Universit\'e Laval, Pavillion Alexandre-Vachon, 1045 Avenue de la M\'edecine, Qu\'ebec, QC, Canada G1V 0A6
}
\email{cedric.dion.1@ulaval.ca}
\date{}

\begin{document}


\begin{abstract}
Let $p>3$ be a prime number and $K$ be an imaginary quadratic field where $p$ splits. Let $K_\infty$ be the $\Zp^2$-extension of $K$ and let $K_n$ be a finite subextension of $K_\infty/K$. Let $E$ be an elliptic curve with good ordinary reduction at $p$. Under some hypotheses, we show that the Mazur--Tate element attached to $E$ over $K_n$ by S. Haran generates the Fitting ideal of the dual Selmer group of $E$ over $K_n$.
\end{abstract}

\maketitle

\section{Introduction}

Let $E$ be elliptic curve with good ordinary reduction at an odd prime $p$ and consider $\QQ_\infty$, the cyclotomic $\Zp$-extension of $\QQ$. Let $\Sigma$ be a finite set of places of $\QQ$ including $p$, $\infty$ and the bad reduction primes of $E$. If $\QQ^{(n)}/\QQ$ is a finite subextension inside $\QQ_\infty$, we attach to $E$ its Selmer group over $\QQ^{(n)}$ defined by the kernel
$$
\Sel(\QQ^{(n)},E[p^\infty]) \colonequals \ker \left( H^1(\QQ_\Sigma/\QQ^{(n)},E[p^\infty]) \to \prod_{v}\frac{H^1(\QQ_v^{(n)},E[p^\infty])}{E(\QQ_v^{(n)})\otimes \Qp/\Zp} \right)
$$
where $\QQ_\Sigma$ is the maximal extension of $\QQ$ unramified outside $\Sigma$, $v$ runs over all the finite places of $\QQ^{(n)}$ above the places in $\Sigma$ and $\QQ_v^{(n)}$ is the completion of $\QQ^{(n)}$ at $v$. On the analytic side, let $f$ be the modular form attached to $E$ by modularity. Denote by $\mu_{p^n}$ the set of $p^n$-th roots of unity and consider the Galois group $G^\prime_{n+1} \colonequals \Gal(\QQ(\mu_{p^{n+1}})/\QQ)/\{\pm 1 \} \cong (\ZZ/p^{n+1}\ZZ)^\times  /\{\pm 1 \}$. Let $\sigma_a$ denote the element corresponding to $a \in (\ZZ/p^{n+1}\ZZ)^\times  /\{\pm 1 \}$. Then, one can define an element in the group ring $\Zp[G^\prime_{n+1}]$ by the formula
$$
\theta^\prime_{n+1}(f) \colonequals \sum_{a \in (\ZZ/p^{n+1}\ZZ)^\times  /\{\pm 1 \}} \left[ \frac{a}{p^{n+1}} \right]^+ \cdot \sigma_a 
$$
where $\left[ \frac{a}{b} \right]^+$ is a modular symbol scaled by an appropriate choice of period. The Mazur--Tate element $\theta_n(f)$ of $f$ at $\QQ^{(n)}$ is defined to be the image of $\theta^\prime_{n+1}(f)$ inside $\Lambda_n \colonequals \Zp[\Gal(\QQ^{(n)}/\QQ)]$. While the Iwasawa main conjecture relates the Pontryagin dual of the Selmer group of $E$ over $\QQ_\infty$ with the $p$-adic $L$-function of $E$, refined Iwasawa theory aims to relate algebraic and analytic objects over $\QQ^{(n)}$. If $M$ is a finitely presented $\Lambda_n$-module, one studies its Fitting ideal $\mathrm{Fitt}_{\Lambda_n}(M)$ (see section \ref{sec:Fitting} for the definition). Mazur and Tate \cite{MT87} gave the following conjecture:

\begin{conjecture}[Weak Mazur--Tate conjecture]\label{conj:weak}
Assume that $a_p \not\equiv 1 \bmod p$, then
$$
\theta_n(f) \in \mathrm{Fitt}_{\Lambda_n}\left( \Hom_{\Zp}(\Sel(\QQ^{(n)},E[p^\infty]), \Qp/\Zp) \right).
$$
\end{conjecture}

A stronger form of this conjecture formulated by Kurihara \cite{Kur02} asserts that $\theta_n(f)$ in fact generates the whole Fitting ideal.

\begin{conjecture}[Strong Mazur--Tate conjecture]\label{conj:strong}
    Assume that $a_p \not\equiv 1 \bmod p$ and that $p$ does not divide the Tamagawa number of $E$, then
    $$
    (\theta_n(f)) = \mathrm{Fitt}_{\Lambda_n}\left( \Hom_{\Zp}(\Sel(\QQ^{(n)},E[p^\infty]), \Qp/\Zp) \right).
    $$
\end{conjecture}

The strong form of the Mazur--Tate conjecture was proven by Kim and Kurihara in \cite{KK21}.

\begin{theorem}[Kim-Kurihara, Strong Mazur--Tate conjecture]
    Let $E$ be an elliptic curve over $\QQ$ with good ordinary reduction at an odd prime $p$. Let $\overline{\rho}$ be the residual Galois representation attached to $E$. Assume that $\overline{\rho}$ is surjective if $E$ is non-CM. Assume that $a_p(E) \not\equiv 1 \bmod p$ and that $p$ does not divide the Tamagawa number of $E$. Further assume that the Iwasawa main conjecture over $\QQ_\infty$ holds. Then,
    $$
    (\theta_n(f)) = \mathrm{Fitt}_{\Lambda_n}\left( \Hom_{\Zp}(\Sel(\QQ^{(n)},E[p^\infty]), \Qp/\Zp) \right).
    $$
\end{theorem}

In this note, we investigate Mazur--Tate type conjectures when $\QQ$ is replaced by an imaginary quadratic field $K$ and the cyclotomic $\Zp$-extension $\QQ_\infty$ is replaced by $K_\infty$, the unique $\Zp^2$-extension of $K$. The setup is the following: We fix a prime number $p>3$ and still consider an elliptic curve $E$ defined over $\QQ$ with good ordinary reduction at $p$ and conductor $N_E$. We let $K$ be an imaginary quadratic field where $(p)=\gp\gq$ splits and we require that both the class number of $K$ and the discriminant of $K$ are corpime to $p$. If $\mathfrak{f}$ is a modulus of $K$, let $K(\mathfrak{f})$ denote the ray class field of $K$ of conductor $\mathfrak{f}$. By using an adelic description of modular symbols, Haran \cite{Han87} constructed analogues of Mazur--Tate elements at $K(\gp^{n+1}\gq^{m+1})$ for integers $n,m \geq -1$. These elements will be denoted by $\Theta_{n,m}$ and live in the group ring $\Zp[\Gal(K(\gp^{n+1}\gq^{m+1})/K)/\Gal(K(p)/K)]$ after scaling by a period. Let $K^\prime(\gp^{n+1}\gq^{m+1})$ be the subfield of $K(\gp^{n+1}\gq^{m+1})$ fixed by $\Gal(K(p)/K)$. Put $\Lambda \colonequals \Zp \llbracket \Gal(K_\infty/K) \rrbracket$. To link the Mazur--Tate element $\Theta_{n,m}$ with the Selmer group of $E$ over $K^\prime(\gp^{n+1}\gq^{m+1})$, we will need a two-variable Iwasawa main conjecture.

\begin{conjecture}[Iwasawa main conjecture]\label{conj:IMC}
    Let $L_p^\mathrm{Hi}(E/K)$ be Hida's two-variable $p$-adic $L$-function attached to $E$ over $K_\infty$. Then, $\Sel(K^\prime(\gp^{n+1}\gq^{m+1}),E[p^\infty])$ is a cotorsion $\Lambda$-module and
    $$
    (L_p^\mathrm{Hi}(E/K)) = \mathrm{Char}_\Lambda \left( \Hom_{\Zp}(\Sel(K_\infty,E[p^\infty]),\Qp/\Zp)\right)
    $$
    as ideals in $\Lambda$.
\end{conjecture}

Let $\Lambda_{n,m}\colonequals \Zp[\Gal(K^\prime(\gp^{n+1}\gq^{m+1})/K)]$. Our theorem is the following:

\begin{theorem}
    Assume that the Iwasawa main conjecture \ref{conj:IMC} holds. Also suppose that $a_p \not\equiv 1 \bmod p$, $E(K)[p]$ is trivial and $E(K_v)[p]$ is trivial for every prime $v$ where $E$ has bad reduction. Then, for $n+m \geq 1$, we have
    $$
    (\Theta_{n,m}) = \mathrm{Fitt}_{\Lambda_{n,m}\otimes \Qp} \left(  \Hom_{\Zp}(\Sel(K^\prime(\gp^{n+1}\gq^{m+1}),E[p^\infty]),\Qp/\Zp) \right)
    $$
    as ideals in $\Lambda_{n,m}\otimes \Qp$.
\end{theorem}

We end this section by reviewing previous works on Mazur--Tate type conjectures. In their work \cite{KK21}, Kim and Kurihara also showed that a modified version of the weak Mazur--Tate conjecture, namely that
$$(\theta_n(f),\nu_{n,n-1}(\theta_{n-1}(f)) \subseteq \mathrm{Fitt}_{\Lambda_n}\left( \Hom_{\Zp}(\Sel(\QQ^{(n)},E[p^\infty]), \Qp/\Zp) \right),$$
also holds for elliptic curves with supersingular reduction at $p$ under mild hypotheses. Here, $\nu_{n,n-1}$ is a trace map $\Lambda_{n-1}\to \Lambda_n$. Moreover, they give sufficient conditions so that the inclusion becomes an equality. Bley and Marcias Castillo \cite[Theorem 2.12]{BMC17} also showed a form of conjecture \ref{conj:strong} for abelian varieties over a cyclic extension of a base number field. Emerton, Pollack and Weston \cite{EPW22} proved a form of conjecture \ref{conj:weak} for Mazur--Tate elements coming from newforms using the $p$-adic Langland correspondance and Kato's zeta element. Kataoka also proved conjecture \ref{conj:weak} over more general number field by using an equivariant point of view in his PhD thesis. C.H. Kim proved an anticyclotomic analogue of conjecture \ref{conj:weak} for $p$-ordinary modular forms over a large class of cyclic ring class extensions of an imaginary quadratic field \cite{kim18} and for supersingular elliptic curves along finite layers of the anticyclotomic $\Zp$-extension of an imaginary quadratic field \cite{kim24}. See also \cite{BKS21} and the works of Ota \cite{Ota18}, \cite{Ota23}.

\textbf{Acknowledgements:} The author is deeply grateful to his PhD advisor Antonio Lei for suggesting the problem and for numerous discussions on the topic. The author also thanks Kâzım Büyükboduk, Bastien Desrochers, Sören Kleine, David Loeffler and Katharina Müller for helpful discussions. This research is supported by the Fond de Recherche du Québec Nature et Technologie's 3rd cycle scholarship.

\section{Notations}
Fix an odd prime number $p>3$. Let $K$ be an imaginary quadratic number field where $(p)=\gp \gq$ splits. Let $K_\infty$ be the compositum of all $\Zp$-extensions of $K$ and let $\Gamma \colonequals \Gal (K_\infty / K)\cong \Zp^2$. We will assume that the class number of $K$ is coprime to $p$. Given a modulus $\mathfrak{f}$ of $K$, we denote by $K(\mathfrak{f})$ its ray class field of conductor $\mathfrak{f}$. Let $S=\{\gp,\gq\}$ denote the set of primes above $p$ in $K$. If $\Sigma$ is any finite set of places of $K$, let $G_{K,\Sigma}$ be the Galois group of the maximal algebraic extension of $K$ unramified outside $\Sigma$. For $v \in S$, let $G_v$ denote the Galois group of the extension $K(v^\infty)\cap K_\infty /K$ and note that $\Gamma \cong G_\gp \times G_{\gq}$. We will use the symbol $\Lambda$ for the Iwasawa algebra $\Zp \llbracket \Gamma \rrbracket \colonequals \varprojlim_H \Zp [\Gamma / H]$ where $H$ runs through the normal open subgroups of $\Gamma$. We fix $\gamma_\gp$ (resp. $\gamma_{\gq}$) a topological generator of $G_\gp$ (resp. $G_{\gq}$). Then, the Iwasawa algebra $\Lambda$ may be identified with the power series ring $\Zp \llbracket \gamma_\gp -1, \gamma_{\gq}-1 \rrbracket$. For an integer $n \geq 0$, define $\omega_n(1+X)\colonequals (1+X)^{p^n}-1$. For an integer $n\geq 1$, define $\Phi_{n}(1+X) \colonequals \sum_{k=0}^{p-1}(1+X)^{kp^{n-1}}$ and $\Phi_0(1+X)\colonequals 1+X$. Let $\iota : \Lambda \to \Lambda$ be the involution induced by $\sigma \mapsto \sigma^{-1}$ for $\sigma \in \Gamma$. If $M$ is any $\Lambda$-module, we write $M^\iota$ for the same module $M$ but with $\Lambda$-action given by $\lambda \cdot m \colonequals \iota(\lambda)m$ for all $m \in M$ and $\lambda \in \Lambda$. Another equivalent definition can be given by setting $M^\iota \colonequals M \otimes_{\Lambda,\iota}\Lambda$. This is a $\Lambda$-module with operations given by $\delta(m \otimes \delta^\prime) = m\otimes \delta\delta^\prime$ and $(\delta m)\otimes \delta^\prime = m\otimes \iota(\delta)\delta^\prime$ for all $\delta,\delta^\prime \in \Lambda, m \in M$. Let $\Delta$ be a finite group and $\eta$ a character of $\Delta$. Suppose that $p \nmid |\Delta |$ and let $\mathcal{O}_\Delta$ be the $p$-adic completion of the extension of $\Z$ obtained be adding the $|\Delta|$-th roots of unity. We let $e_\eta \colonequals \frac{1}{|\Delta|}\sum_{\sigma \in \Delta}\eta(\sigma)^{-1}\cdot \sigma \in \mathcal{O}_\Delta [\Delta]$ be the idempotent element corresponding to $\eta$.

Let $E$ be an elliptic curve defined over $\Q$ with good ordinary reduction at $p$ and with conductor $N_E$ not divisible by $p$. As a representation of $G_{\Qp}$, the $p$-adic Tate module $T \colonequals T_p E = \varprojlim_n E[p^n]$ is reducible and we have an exact sequence of $\Zp[G_{\Qp}]$-modules
$$
0 \to T^+ \to T \to T^- \to 0
$$
where $T^+$ and $T^-$ are free of rank one over $\Zp$ and $T^-$ is unramified. Let $A\colonequals \Hom_\mathrm{cts}(T,\mu_{p^\infty})\cong E[p^\infty]$ where $\mu_{p^\infty}$ is the set of $p$-power roots of unity and let $A^- \colonequals \Hom_{\mathrm{cts}}(T^-,\mu_{p^\infty})$. We let $a_p$ be the quantity defined by $p+1-|E(\mathbb{F}_p)|$. If $v$ is a place of $K$, $c_v$ will denote the local Tamagawa number of $E$ at $v$ which is defined by the index $[E(K_v):E^0(K_v)]$ where $E^0(K_v)$ is the subgroup of points which have good reduction. We let $\mathrm{Tam}(E)$ denote the product of the local Tamagawa numbers $\prod_v c_v$.

If $L$ is a number field or a local field, we write $G_L$ for the Galois group $\Gal (\overline{L}/L)$ where $\overline{L}$ is an algebraic closure of $L$.

\section{Iwasawa main conjecture for $\mathrm{GL}_2$}

In this section, we review the formulation of the Iwasawa main conjecture that was proven by Castella and Wan in \cite{CW22}.

\subsection{Selmer groups}

We first review the construction of the Selmer groups appearing in the main conjecture.

Let $F$ be a finite extension of $K$ contained in $K_\infty$ and let $\Sigma$ be a finite set of places of $F$ containing the infinite places and the primes dividing $N_Ep$. Let $G_{F,\Sigma}$ be the Galois group of the maximal extension of $F$ unramified outside $\Sigma$. 

\begin{definition}
    Define the (strict) Greenberg Selmer group $\Sel_\mathrm{Gr}(F,A)$ over $F$ by
    $$
    \Sel_\mathrm{Gr}(F,A) \colonequals \ker \left( H^1(G_{F,\Sigma},A) \to \prod_{w \nmid p} H^1(I_w,A) \times \prod_{w|p}H^1(F_w,A^-) \right)
    $$
    where $w$ runs over the places of $F$ and $I_w$ is the inertia subgroup at $w$ inside $\Gal(\overline{F_w}/F_w)$.
    Define the (strict) Greenberg Selmer group $\Sel_\mathrm{Gr}(K_\infty,A)$ over $K_\infty$ by
    $$
    \Sel_\mathrm{Gr}(K_\infty,A) \colonequals \ker \left( H^1(K_\infty,A) \to \prod_{w \nmid p}H^1(I_w,A) \times \prod_{w | p}H^1(K_{\infty,w}, A^-) \right)
    $$
    where $w$ runs over the places of $K_\infty$ and $I_w$ is the inertia subgroup at $w$ inside $\Gal(\overline{K_{\infty,w}}/K_{\infty,w})$.
\end{definition}

We shall consider the Pontryagin dual of the previously defined Selmer group 
$$
\mathcal{X}_\mathrm{Gr}(K_\infty,A) \colonequals \Hom_{\Zp} (\Sel_\mathrm{Gr}(K_\infty,A), \Qp/\Zp).
$$

One may also consider the Kummer sequence
$$
0 \to E(F_w)\otimes \Qp/\Zp \xrightarrow{k_w} H^1(F_w,A) \to H^1(F_w, E(\overline{F_w}))_{\text{$p$-primary}} \to 0
$$
and define the classical Selmer group $\Sel_\mathrm{class}(F,A)$ over $F$ by
$$
\Sel_\mathrm{class}(F,A) \colonequals \ker \left( H^1(F,A) \to \prod_{w|p}\frac{H^1(F_w,A)}{\mathrm{Im}(k_w)} \right).
$$
The classical Selmer over $K_\infty$ is then given by the direct limit $$\Sel_\mathrm{class}(K_\infty,A) \colonequals \varinjlim_{K\subseteq F \subseteq K_\infty} \Sel_\mathrm{class}(F,A)$$ where $F$ runs over the finite subextensions of $K_\infty/K$.
The discussion of \cite[\S 2]{Gre89} shows that
\begin{equation}\label{eqn:Kummer=strict}
    \Sel_\mathrm{Gr}(F,A)=\Sel_\mathrm{class}(F,A).
\end{equation}

The Greenberg Selmer group satisfies the following control theorem:

\begin{proposition}\label{prop:Control}
Let $F$ be a finite extension contained in $K_\infty/K$. Assume that $E(K)[p]$ is trivial and that $a_p \not\equiv 1 \bmod p$. For all prime $v$ where $E$ has bad reduction, assume that $E(K_v)[p]$ is trivial. Then, 
$$
\Sel_{\mathrm{Gr}}(K_\infty,A)^{\Gal(K_\infty/F)} \cong \Sel_\mathrm{Gr}(F,A).
$$
\end{proposition}

\begin{proof}
Let $F$ be a finite extension contained in $K_\infty/K$. Note that the condition $a_p \not\equiv 1 \bmod p$ is equivalent to $E(\mathbf{F}_p)[p] = \{0\}$. Under the assumptions of the proposition, \cite[Proposition 5.6]{Gre03} gives
$$
\Sel_\mathrm{class}(K_\infty,A)^{\Gal (K_\infty/F)} \cong \Sel_\mathrm{class}(F,A).
$$
However, it follows from \eqref{eqn:Kummer=strict} that both definition of Selmer groups agree and so the result follows.
\end{proof}

\subsection{$p$-adic $L$-functions}

Let $f$ be the weight $2$ modular form of level $\Gamma_0(N_E)$ associated to $E$ by modularity. Remark that the Fourier coefficients of $f$ are defined over $\QQ$. By our ordinarity hypothesis, there is a unique unit root of the Hecke polynomial $X-a_p(E)+p$ which we denote by $\alpha$. Let $f_\alpha$ be the $p$-stabilization of $f$ associated to the root $\alpha$. Then, Hida defined a three-variable $p$-adic $L$-function $L_p^{Hi}(\mathbf{f}/K)$ where $\mathbf{f}$ is a Hida family over $\mathbb{I}$, some flat extension of a one-variable Iwasawa algebra (see \cite[Section 2.4]{CW22} for a detailed construction). By taking $\mathbf{f}$ to be a Hida family passing through $f_\alpha$, we may specialize $L_p^{Hi}(\mathbf{f}/K)$ to a two-variable $p$-adic $L$-function $L_p^{Hi}(f_\alpha/K)$. Since $f_\alpha$ is defined over $\Qp$, $L_p^{Hi}(f_\alpha/K)$ is an element of $\Lambda$.

\subsection{The main conjecture}

We now state the main result of \cite{CW22}. Let $\rho_{f_\alpha}:G_\QQ\to \mathrm{GL}_2(\Qp)$ be the Galois representation associated to $f_\alpha$ and $\overline{\rho}_{f_\alpha}:G_\QQ\to \mathrm{GL}_2(\ZZ/p\ZZ)$ be the residual representation. Suppose that the quadratic field $K$ has discriminant coprime to $N_Ep$. Write $N_E=N^+N^-$ where $N^+$ (resp. $N^-$) is divisible only by primes which are split (resp. inert) in $K$. We will assume that $K$ satisfies the generalized Heegner hypothesis relative to $N_E$:
$$
\text{$N^-$ is the squarefree product of an even number of primes.}
$$

\begin{theorem}\label{thm:IMC}
    Suppose that $N_E$ is squarefree, $N^- \neq 1$ and that $\overline{\rho}_f$ is ramified at every prime dividing $N^-$. Also assume that $\mathbb{I}$ is regular. Then, $X_\mathrm{Gr}(K_\infty,A)$ is $\Lambda$-torsion and 
    $$
    \mathrm{Char}_\Lambda (X_\mathrm{Gr}(K_\infty,A)) = (L_p^\mathrm{Hi}(f_\alpha/K))
    $$
    as ideals in $\Lambda \otimes \Qp$.
\end{theorem}

\section{Two-variable Mazur--Tate elements}

We review the main result of Haran \cite{Han87} on the construction of Mazur--Tate elements for $E$ over $K_\infty$ by following the exposition in \cite{Loe14}.

We write $G_\mathfrak{f} \colonequals \Gal (K(\mathfrak{f})/K)$ for the ray class group of $K$ modulo $\mathfrak{f}$ and $G_1 \colonequals \Gal(K(1)/K)$ for the Galois group of the Hilbert class field of $K$. If $\mathfrak{m}$ and $\mathfrak{f}$ are two moduli such that $\mathfrak{m}|\mathfrak{f}$, we have a natural surjection $G_\mathfrak{f}\twoheadrightarrow G_\mathfrak{m}$. We let $\mathrm{Norm}_\mathfrak{m}^\mathfrak{f}:\Zp[G_\mathfrak{f}]\to \Zp[G_\mathfrak{m}]$ be the induced map on the group rings. We define the trace map by \begin{align*}
\mathrm{Tr}_{\mathfrak{m}}^\mathfrak{f}:\Zp [G_\mathfrak{m}]&\to \Zp [G_\mathfrak{f}]\\
x &\mapsto \sum_{y \in \Zp [G_\mathfrak{f}]\atop{\mathrm{Norm}^{\mathfrak{f}}_\mathfrak{m}(y)=x}} y.
\end{align*}
Let $v \in \{\gp,\gq\}$. In the special case where $\mathfrak{m}=v^{n-1}$ and $\mathfrak{f}=v^n$ with $n\geq 2$, remark that $\mathrm{Tr}_{\mathfrak{m}}^\mathfrak{f}$ is given by multiplication by $\Phi_{n-1}(\gamma_v)$. Indeed, preimages of $x$ under the norm map take the form $y\gamma_v^{kp^{n-1}}$ where $k=0,\ldots p-1$ and $y$ is a fixed preimage of $x$. 

Let $\Pi/K$ be the base-change to $K$ of the automorphic representation coming from $f$. Since $p$ is split in $K$, the $v$-th Hecke eigenvalue $a_v(\Pi/K)$ is equal to $a_p$ for both choices $v=\gp,\gq$.

\begin{theorem}\label{thm:Haran}
There exists a rank one $\ZZ$-submodule $L_f$ of $\mathbf{C}$, and for each ideal $\mathfrak{f}$ of $K$ coprime to the conductor of $\Pi$ an element 
$$
\widetilde{\Theta}_\mathfrak{f}(\Pi) \in \Z[G_\mathfrak{f}] \otimes_\Z L_f,
$$
where $G_\mathfrak{f}$ is the ray class group modulo $\mathfrak{f}$, such that the following relations hold:
\begin{enumerate}
\item If $\omega$ is a primitive ray class character of conductor $\mathfrak{f}$, then
$$
\omega(\widetilde{\Theta}_\mathfrak{f}(\Pi)) = \frac{L_\mathfrak{f}(\Pi,\omega,1)}{\tau(\omega)\cdot |\mathfrak{f}|^{1/2}\cdot (4\pi)^2}
$$
where $L_\mathfrak{f}(\Pi,\omega,s)$ denotes the $L$-function of $\Pi$ twisted by $\omega$, without the Euler factors at $\infty$ or at primes dividing $\mathfrak{f}$, and $\tau(\omega)$ is the Gauss sum (normalized so that $|\tau(\omega)|=1$);
\item For a prime $v$ not dividing the conductor of $\Pi$, we have
$$
\mathrm{Norm}_{\mathfrak{f}}^{\mathfrak{f}v} \widetilde{\Theta}_{\mathfrak{f}v}(\Pi) = \begin{cases} \left( a_v(\Pi)-\sigma_v-\sigma_v^{-1}\right) \widetilde{\Theta}_{\mathfrak{f}}(\Pi) & \text{if $v\nmid \mathfrak{f}$}, \\ a_v(\Pi)\widetilde{\Theta}_{\mathfrak{f}}-\mathrm{Tr}_{\mathfrak{f}/v}^\mathfrak{f}\widetilde{\Theta}_{\mathfrak{f}/v}(\Pi) & \text{if $v | \mathfrak{f}$}, \end{cases}
$$
where $\sigma_v$ denotes the class of $v$ in $G_\mathfrak{f}$.
\end{enumerate}
\end{theorem}

\begin{proof}
This is the main result of \cite{Han87} combined with \cite[Proposition 4.3]{Loe14} for the fact that $L_f$ has rank one. 
\end{proof}

Since we only consider $\Pi$ in this note, we remove it from the notation and write $\widetilde{\Theta}_{\mathfrak{f}}$ for the Mazur--Tate element of theorem \ref{thm:Haran}.

\begin{definition}
A $S$-refinement of $\Pi$ is a choice of roots $\alpha_v \in \overline{\QQ}$ for $v\in S$ of the Hecke polynomial $X^2-a_v(\Pi)X+p \in \QQ[X]$.
\end{definition}

\begin{definition}\label{dfn:Haran elements}
When $p$ divides the ideal $\mathfrak{f}$, we define the $S$-refined Mazur--Tate element by
$$
\widetilde{\Theta}_{\mathfrak{f}}^S \colonequals \alpha_{\mathfrak{p}}^{-\ord_\gp (\mathfrak{f})}\alpha_{\mathfrak{q}}^{-\ord_{\gq} (\mathfrak{f})} \left( \widetilde{\Theta}_{\mathfrak{f}} - \alpha_\gp^{-1}\mathrm{Tr}_{\mathfrak{f}/\gp}^{\mathfrak{f}}\widetilde{\Theta}_{\mathfrak{f}/\gp}- \alpha_{\gq}^{-1}\mathrm{Tr}_{\mathfrak{f}/\gq}^{\mathfrak{f}}\widetilde{\Theta}_{\mathfrak{f}/\gq}+\alpha_\gp^{-1}\alpha_{\gq}^{-1}\mathrm{Tr}_{\mathfrak{f}/\gp\gq}^{\mathfrak{f}}\widetilde{\Theta}_{\mathfrak{f}/\gp\gq} \right).
$$

If $\gq$ divides $\mathfrak{f}$, but $\gp$ does not, we put
$$
\widetilde{\Theta}_\mathfrak{f}^S \colonequals \alpha_\gp^{-1}\alpha_\gq^{-\mathrm{ord}_\gq(\mathfrak{f})}(\alpha_\gp-\sigma_\gp-\sigma_\gp^{-1}) \left( \widetilde{\Theta}_\mathfrak{f} - \alpha_\gq^{-1}\mathrm{Tr}_{\mathfrak{f}/\gq}^{\mathfrak{f}}\widetilde{\Theta}_{\mathfrak{f}/\gq} \right)
$$
and if $\gp$ divides $\mathfrak{f}$, but $\gq$ does not, we put
$$
\widetilde{\Theta}_\mathfrak{f}^S \colonequals \alpha_\gp^{-\mathrm{ord}_\gp(\mathfrak{f})}\alpha_\gq^{-1}(\alpha_\gq-\sigma_\gq-\sigma_\gq^{-1}) \left( \widetilde{\Theta}_\mathfrak{f} - \alpha_\gp^{-1}\mathrm{Tr}_{\mathfrak{f}/\gp}^{\mathfrak{f}}\widetilde{\Theta}_{\mathfrak{f}/\gp} \right).
$$

Finally, in the case where nether $\gp$ nor $\gq$ divides $\mathfrak{f}$, we put
$$
\widetilde{\Theta}_\mathfrak{f}^S \colonequals \left( \prod_{v \in S} (1-\alpha_v^{-1}\sigma_v)(1-\alpha_v^{-1}\sigma_v^{-1}) \right) \widetilde{\Theta}_\mathfrak{f}.
$$
\end{definition}
One can check that 
$$
\mathrm{Norm}_{\mathfrak{f}}^{\mathfrak{f}v}\left( \widetilde{\Theta}_{\mathfrak{f}v}^S \right) = \widetilde{\Theta}_{\mathfrak{f}}^S
$$
for any $v \in S$. Then the elements $\{ \widetilde{\Theta}_{\mathfrak{f}}^S : \mathfrak{f} | p^\infty \}$ form a ray class distribution modulo $p^\infty$ with values in $\QQ(\alpha_\gp,\alpha_{\gq}) \otimes_{\mathcal{O}_{\QQ(\alpha_\gp,\alpha_{\gq})}} L_f$ and growth parameter $(\alpha_\gp,\alpha_{\gq})$ in the sense of Loeffler \cite{Loe14}. Fix an embedding $\QQ(\alpha_\gp,\alpha_{\gq}) \hookrightarrow \overline{\Qp}$ which is the same as fixing a prime $\mathfrak{e}$ above $p$. Then, the completion $\QQ(\alpha_\gp,\alpha_{\gq})_\mathfrak{e}$ is equal to $\Qp(\alpha_\gq,\alpha_\gq)$. But it follows by Hensel's lemma that $\alpha_\gp,\alpha_\gq \in \Qp$, hence $\QQ(\alpha_\gp,\alpha_{\gq})_\mathfrak{e}=\Qp$.

\begin{theorem}\label{thm:Loeffler interpolation}
If we have $r \colonequals \ord_p(\alpha_\gp) < 1$ and $s \colonequals \ord_p(\alpha_{\gq}) < 1$, then there exists a unique distribution $L_{p,\underline{\alpha}}(f) \in D^{(r,s)}(G_{p^\infty},\Qp) \otimes_{\ZZ} L_f$ satisfying the interpolation formula \footnote{The exponent $1/2$ of $|\mathfrak{f}|$ is missing in \cite[Theorem 8]{Loe14}.}
\begin{align*}
L_{p,\underline{\alpha}}(f)(\omega) = \left( \prod_{v \in S} \alpha_v^{-\ord_v (\mathfrak{f})} \right) \cdot \left( \prod_{v \in S, v \nmid \mathfrak{f}}(1-\alpha_v^{-1}\omega(v))(1-\alpha_v^{-1}\omega(v)^{-1})\right)\cdot \frac{L_{\mathfrak{f}}(\Pi,\omega,1)}{\tau(\omega)\cdot |\mathfrak{f}|^{1/2}\cdot (4\pi)^2}
\end{align*}
where
\begin{itemize}
\item $\omega$ is a finite order ray class character on $G_\mathfrak{f}$ with $\mathfrak{f}|p^\infty$,
\item $L_{\mathfrak{f}}(\Pi,\omega,s)$ denotes the $L$-function of $\Pi$ twisted by $\omega$, without the Euler factors at $\infty$ or at primes dividing $\mathfrak{f}$,
\item $\tau(\omega)$ is the Gauss sum.
\end{itemize}
\end{theorem}

\begin{proof}
This is \cite[Theorem 9]{Loe14} applied to the case where $\Pi$ is the base-change to $K$ of the automorphic representation attached to $f$. A priori, $L_{p,\underline{\alpha}}(f)$ takes value in $\QQ(\alpha_\gp,\alpha_{\gq})_\mathfrak{e}\otimes_{\ZZ}L_f$, but the discussion above shows that $L_{p,\underline{\alpha}}(f) \in D^{(r,s)}(G_{p^\infty},\Qp) \otimes_{\ZZ} L_f$.
\end{proof}

In other words, $L_{p,\underline{\alpha}}(f)$ is the unique distribution interpolating the $\widetilde{\Theta}_\mathfrak{f}^S$ when $\mathfrak{f}|p^\infty$. Thus, if $\omega$ is a primitive ray class character of conductor $\mathfrak{f}$ as above, $L_{p,\underline{\alpha}}(f)(\omega) = \widetilde{\Theta}_\mathfrak{f}^S(\omega)$. Under our $p$-ordinary assumption, there is a unique choice of roots $\alpha_{\gp}$ and $\alpha_{\gq}$ (with $\ord_p(\alpha_{\gp})=\ord_p(\alpha_{\gq})=0$) satisfying the hypothesis of theorem \ref{thm:Loeffler interpolation}. We shall make this choice and suppress it from the notation. We obtain a measure $L_p(f) \in \mathcal{M}(G_{p^\infty},\Qp)\otimes_{\Zp}L_f$. This corresponds, under the Amice transform, to an element of $\Zp\llbracket G_{p^\infty} \rrbracket \otimes_{\Zp} L_f$ which we also denote by $L_p(f)$. Suppose that \\
\textbf{(C.N)} $p$ does not divide the class number of $K$.\\
Since $G_{p^\infty} \cong \Gamma \times \Delta$ where $\Delta = \Gal(K(p)/K)$ is a finite group with order not divisible by $p$ by \textbf{(C.N)}, we may define $L_{p,\Delta}(f) \colonequals e_{\chi_0}\cdot L_p(f) \in \Zp\llbracket \Gamma \rrbracket \otimes_{\Zp}L_f$ where $\chi_0$ is the trivial character $\Delta \to \Qp^\times$. Finally, by choosing a generator $\Omega_{\mathrm{Loe}}$ of $L_f$, we view $L_{p,\Delta}(f)$ as an element of $\Zp \llbracket \Gamma \rrbracket$ and $\widetilde{\Theta}_\mathfrak{f}$ as an element of $\Zp[G_\mathfrak{f}]$. For any modulus $\mathfrak{f} | p^\infty$, let $K^\prime(\mathfrak{f})$ be the largest Galois subextension of $K(\mathfrak{f})$ with degree a power of $p$ and let $G_\mathfrak{f}^\prime$ be $\Gal (K^\prime (\mathfrak{f})/K)$, i.e. $\Gal(K(\mathfrak{f})/K) \cong G_\mathfrak{f}^\prime \times \Delta$. 
\begin{definition}
We define the Mazur--Tate element $\Theta_{n,m}^S$ (resp. $\Theta_{n,m}$) at $K^\prime(\gp^{n+1}\gq^{m+1})$ to be the projection to $\Zp[G^\prime_{\gp^{n+1}\gq^{m+1}}]$ of $\widetilde{\Theta}_{\gp^{n+1}\gq^{m+1}}^S$ (resp. $\widetilde{\Theta}_{\gp^{n+1}\gq^{m+1}}$).
\end{definition}

\begin{lemma}\label{lem:groupring unit}
The element $1-\alpha_v^{-1}\sigma_v$ is invertible in $\Zp[G_1]$. Furthermore, the element $1-\alpha_v^{-1}\sigma_v-\alpha_v^{-1}\sigma_v^{-1}$ is also invertible in $\Zp[G_{v^n}]$ ($n\geq 1$).
\end{lemma}

\begin{proof}
    Say that $G_1$ has $k$ elements $\{g_1,\ldots, g_k\}$. Write $\{\chi_1,\ldots,\chi_k\}$ for the $k$ characters on $G_1$. A general element of $\Zp[G_1]$ takes the form $\sum_{i=1}^ka_ig_i$ with $a_i \in \Zp$. Thus, we want to show that the equation with unknowns $a_i$ ($i=1,\ldots,k$)
    $$
    (1-\alpha_v^{-1}\sigma_v)\cdot \sum_{i=1}^ka_ig_i =1
    $$
    has a solution. Put in matrix form, this system of equations forms a $G_1$-circulant matrix (see \cite[\S 1.3]{johnson2019group} for a reference on the topic). In \cite[discussion after Corollary 2.4]{johnson2019group}, it is shown that the determinant of such matrices is given by 
    $$
    \mathrm{det}=\prod_{i=1}^k \left(1-\alpha_v^{-1}\chi_i(\sigma_v)\right).
    $$
    Since the $\chi_i(\sigma_v)$ are $k$-th roots of unity, the product is non-zero. Hence, an inverse certainly exists in $\Qp[G_1]$. To show integrality, notice that the solutions $a_i$ can be written as $a_i=b_i/\mathrm{det}$ for some $b_i$ in $\Zp$. Thus, if $p$ divided $\mathrm{det}$, it would also divide $1-\alpha_v^{-1}\sigma_v$ which is not the case. This shows that $(1-\alpha_v^{-1}\sigma_v)^{-1}\in \Zp[G_1]$.

    For $1-\alpha_v^{-1}\sigma_v-\alpha_v^{-1}\sigma_v^{-1}$, remark that $G_{v^n}$ has $k(p-1)p^{n-1}$ elements. Say that these elements are $\{g_1,\ldots, g_{k(p-1)p^{n-1}}\}$ with associated characters $\{\chi_1,\ldots, \chi_{k(p-1)p^{n-1}}\}$. By \textit{loc. cit.}, the determinant of the system of equations we now want to solve is given by
    $$
    \prod_{i=1}^{k(p-1)p^{n-1}}(1-\alpha_v^{-1}(\chi_i(\sigma_v)-\chi_i(\sigma_v)^{-1})).
    $$
    This last expression is zero if and only if one of the $1-\alpha_v^{-1}(\chi_i(\sigma_v)-\chi_i(\sigma_v)^{-1})$ is zero. This equivalent to 
    $$
    \alpha_v=\chi_i(\sigma_v)-\chi_i(\sigma_v)^{-1}.
    $$
    Fix an embedding $\iota_p:\overline{\QQ} \hookrightarrow \mathbf{C}_p$. Since $a_v \neq 0$, $\iota_p^{-1}(\alpha_v)$ can not be a purely imaginary complex number. But, $\iota_p^{-1}(\chi_i(\sigma_v)-\chi_i(\sigma_v)^{-1})$ is purely imaginary, hence we can not have equality. We conclude that $1-\alpha_v^{-1}\sigma_v-\alpha_v^{-1}\sigma_v^{-1}$ is invertible in $\Qp[G_{v^{n}}]$. By the same argument as before, it is also invertible in $\Zp[G_{v^{n}}]$.
\end{proof}

\begin{lemma}\label{lem:theta1}
For $n,m \geq 0$, we have $\widetilde{\Theta}_{\gp^n\gq^m} = C(n,m)\widetilde{\Theta}_{\gp^n\gq^m}^S$ where $C(n,m)\in \Zp[G_{\gp^{n}\gq^{m}}]$.
\end{lemma}
\begin{proof}
We show this by strong two-variable induction on $n$ and $m$.

\textbf{Base case:} By definition \ref{dfn:Haran elements}, we have
$$
\widetilde{\Theta}_{(1)} \cdot \left( \prod_{v \in S} (1-\alpha_v^{-1}\sigma_v)(1-\alpha_v^{-1}\sigma_v^{-1}) \right) = \widetilde{\Theta}_{(1)}^S
$$
as elements of $\Zp[G_1]$. By lemma \ref{lem:groupring unit}, $C(0,0)^{-1}\colonequals \prod_{v \in S} (1-\alpha_v^{-1}\sigma_v)(1-\alpha_v^{-1}\sigma_v^{-1})$ is invertible, hence $\widetilde{\Theta}_{(1)}=C(0,0)\cdot \widetilde{\Theta}^S_{(1)}$.

\textbf{Induction step:} Suppose the result is true for $\widetilde{\Theta}_{\gp^n\gq^m}$ for all $n,m$ such that $n+m<l$. We want to show it also holds for all $n,m$ such that $n+m=l+1$. We distinguish two cases: if both $n,m$ are non zero or one of $m$ or $n$ is zero. If $m\neq 0$ and $n\neq 0$, we have by definition that
\begin{align*}
\widetilde{\Theta}_{\gp^{n}\gq^m}^S = \alpha_\gp^{-n}\alpha_\gq^{-m}\Big( \widetilde{\Theta}_{\gp^{n}\gq^m} -\alpha_\gp^{-1}\mathrm{Tr}_{\gp^{n-1}\gq^m}^{\gp^{n}\gq^m}\widetilde{\Theta}_{\gp^{n-1}\gq^m} &- \alpha_\gp^{-1}\mathrm{Tr}_{\gp^{n}\gq^{m-1}}^{\gp^{n}\gq^m}\widetilde{\Theta}_{\gp^{n}\gq^{m-1}} \\
&+ \alpha_\gp^{-1}\alpha_\gq^{-1}\mathrm{Tr}_{\gp^{n-1}\gq^{m-1}}^{\gp^{n}\gq^m}\widetilde{\Theta}_{\gp^{n-1}\gq^{m-1}}\Big).
\end{align*}
By the induction hypothesis, we may replace the various $\widetilde{\Theta}_{\gp^a\gq^b}$ by $C(a,b)\widetilde{\Theta}_{\gp^a\gq^b}^S$ in the right hand side of the last expression. Since $\mathrm{Tr}_{\mathfrak{e}}^\mathfrak{f}\widetilde{\Theta}_\mathfrak{e}^S$ is a multiple of $\widetilde{\Theta}_\mathfrak{f}^S$ by some element of $\Zp[G_\mathfrak{f}]$, we find $\widetilde{\Theta}_{\gp^n\gq^m}=C(n,m)\widetilde{\Theta}_{\gp^n\gq^m}^S$ with
\begin{equation}\label{eqn:C(n,m)}
C(n,m) \colonequals \alpha_\gp^n\alpha_\gq^m+\alpha_\gp^{-1}h_1C(n-1,m)+\alpha_\gq^{-1}h_2C(n,m-1)-\alpha_\gp^{-1}\alpha_\gq^{-1}h_3C(n-1,m-1)
\end{equation}
where $h_1,h_2$ and $h_3$ are some elements of $\Zp[G_{\gp^n\gq^m}]$. For the second case, assume without loss of generality that $m=0$. By definition,
\begin{align*}
\widetilde{\Theta}_{\gp^n}^S = \alpha_\gp^{-n}(1-\alpha_\gq^{-1}\sigma_\gq-\alpha_\gq^{-1}\sigma_\gq^{-1})\left( \widetilde{\Theta}_{\gp^n}-\alpha_\gp^{-1}\mathrm{Tr}_{\gp^{n-1}}^{\gp^n}\widetilde{\Theta}_{\gp^{n-1}}\right).
\end{align*}
By the induction hypothesis, $\mathrm{Tr}_{\gp^{n-1}}^{\gp^n}\widetilde{\Theta}_{\gp^{n-1}} = C(n-1,0)\mathrm{Tr}_{\gp^{n-1}}^{\gp^n}\widetilde{\Theta}_{\gp^{n-1}}^S$ and thus $\mathrm{Tr}_{\gp^{n-1}}^{\gp^n}\widetilde{\Theta}_{\gp^{n-1}} = C(n-1,0)h_4\widetilde{\Theta}_{\gp^{n}}$ for some element $h_4$ of $\Zp[G_{\gp^n}]$. It follows that $\widetilde{\Theta}_{\gp^n}=C(n,0)\widetilde{\Theta}_{\gp^n}^S$ where
$$
C(n,0) \colonequals \alpha_\gp^n(1-\alpha_\gq^{-1}\sigma_\gq-\alpha_\gq^{-1}\sigma_\gq^{-1})^{-1}+\alpha_\gp^{-1}C(n-1,0)h_4.
$$
Note that we were able to invert $1-\alpha_\gq^{-1}\sigma_\gq-\alpha_\gq^{-1}\sigma_\gq^{-1}$ by lemma \ref{lem:groupring unit}. This complete the induction step.
\end{proof}

\begin{proposition}\label{prop:B(n,m)}
    For $n+m \geq 1$, we have $\Theta_{n,m}=B(n,m)\Theta_{n,m}^S$ where $B(n,m) \in \Zp[G_{\gp^{n+1}\gq^{m+1}}^\prime]^\times$.
\end{proposition}

\begin{proof}
    To show this, we need to study more carefully the $h_i$'s appearing in the proof of lemma \ref{lem:theta1}. Consider $x \in G_{\gp^{n-1}\gq^m}$. Let $y$ be a preimage of $x$ in $G_{\gp^n\gq^m}$ and let $\gamma_n$ be a generator of $G_{\gp^n}/G_\gp\cong \ZZ/p^{n-1}\ZZ$. If we suppose that $n\geq 2$, we have $$\mathrm{Tr}_{\gp^{n-1}\gq^m}^{\gp^n\gq^m} x = \sum_{i=0}^{p-1} y\gamma_n^{p^{n-2}i} = y\cdot \Phi_{n-1}(\gamma_n).$$
    Since $\widetilde{\Theta}_{\gp^n\gq^m}^S$ is a preimage of $\widetilde{\Theta}_{\gp^{n-1}\gq^m}^S$, we find $\mathrm{Tr}_{\gp^{n-1}\gq^m}^{\gp^n\gq^m}\widetilde{\Theta}_{\gp^{n-1}\gq^m}^S = \Phi_{n-1}(\gamma_n)\cdot\widetilde{\Theta}_{\gp^{n}\gq^m}^S$. Remark that, under the isomorphism $\Zp\llbracket G_{\gp^\infty}/G_{\gp\gq}\rrbracket \xrightarrow{\sim} \Zp\llbracket X \rrbracket$, $\Phi_{n-1}(\gamma_n)$ is send to $\Phi_{n-1}(1+X)$ which has constant term divisible by $p$. Let $\pi_{n,m}$ be the projection map $G_{\gp^n\gq^m}\to G_{\gp^n\gq^m}/G_{\gp\gq}$. The upshot of this is that $p$ divides the constant term of all the $\pi_{n,m}(h_i)=h_i$'s provided that $n+m\geq 2$. Define $B(n-1,m-1)\colonequals \pi_{n,m} C(n,m)$ (so that we have $\Theta_{n,m}=B(n,m)\Theta_{n,m}^S$ by lemma \ref{eqn:C(n,m)}). Write $b(n,m)$ for the constant term of $B(n,m)$ viewed as an element of $\Zp\llbracket X,Y\rrbracket$ via the isomorphism $\Zp\llbracket G_{p^\infty}/G_{\gp\gq}\rrbracket \xrightarrow{\sim}\Zp\llbracket X,Y\rrbracket$. Then, the constant term of $B(n,m)$ is $$\alpha_\gp^{n-1}\alpha_\gq^{m-1} +p\left( \alpha_\gp^{-1}b(n-1,m)+\alpha_\gq^{-1}b(n,m-1) -p\alpha_\gp^{-1}\alpha_\gq^{-1}b(n-1,m-1)\right)$$
    or
    $$
    \alpha_\gp^{n-1}(1-\alpha_\gq^{-1}\pi_{n-1,0}\sigma_\gq-\alpha_\gq^{-1}\pi_{n-1,0}\sigma_\gq^{-1})^{-1}+p\alpha_\gp^{-1}b(n-1,0)
    $$
    depending on whether or not one of the $n,m$ is zero. In both cases, we find that $B(n,m)$ is a unit.
\end{proof}

\begin{remark}
    In the statement of proposition \ref{prop:B(n,m)}, we need to exclude the pairs $(n,m) \in \{(0,0),(1,0),(0,1)\}$. In these cases, the elements $h_i$ are given by traces of the form $\mathrm{Tr}_{\gp}^{\gp\gq}$. Since $G_{\gp\gq}/G_\gp \cong \ZZ/(p-1)\ZZ$, the trace map is not given by multiplication by a cyclotomic polynomial. It may be possible that the $B(n,m)$ in those cases are still units, but one would need to find a different argument.
\end{remark}

\begin{theorem}\label{thm:stabilized}
    For $n+m\geq 1$, we have $(\Theta_{n,m}) = (\Theta_{n,m}^S)$ as ideals in $\Zp[G_{\gp^{n+1}\gq^{m+1}}^\prime]$.
\end{theorem}

\begin{proof}
    Follows directly by proposition \ref{prop:B(n,m)}.
\end{proof}

\section{Comparing $p$-adic $L$-function}

Recall that the period appearing in the interpolation formula for $L_p^{\mathrm{Hi}}(f_\alpha/K)$ \cite[Theorem 2.1]{BL22} is the Petersson inner product $\langle f,f\rangle_{N_E}$.

\begin{proposition}\label{prop:L-functions}
    We have $(L_{p,\Delta}(f)) = (L_p^{\mathrm{Hi}}(f_\alpha/K))$ as ideals in $\Lambda \otimes \Qp$. Furthermore, if the ratio of the periods $\Omega_{\mathrm{Loe}}/\langle f,f\rangle_{N_E}$ is a $p$-adic unit, then the equality holds in $\Lambda$.
\end{proposition}

\begin{proof}
     Comparing the interpolation formula \ref{thm:Loeffler interpolation} for $L_{p,\Delta}(f)$ and the one for $L_p^{\mathrm{Hi}}(f_\alpha/K)$ \cite[Theorem 2.1]{BL22}, one finds that the quotient $L_{p,\Delta}(f)/L_p^{\mathrm{Hi}}(f_\alpha/K)$ is an element of $\Qp$ (remark that the absolute value of root number $\tau(\psi_p)$ of \textit{loc. cit.} is $p^{(n+m)/2}$ whereas the absolute value of the Gauss sum of $\psi$ is $1$). If $\Omega_{\mathrm{Loe}}/\langle f,f\rangle_{N_E} \in \Zp^\times$, then we get that $L_{p,\Delta}(f)/L_p^{\mathrm{Hi}}(f_\alpha/K)\in \Zp^\times$.
\end{proof}

\section{Selmer complexes}

\subsection{Fitting ideals}\label{sec:Fitting}

Let $R$ be a ring and $M$ be a finitely presented $R$-module. Let 
$$
R^s \xrightarrow{h} R^r \to M \to 0
$$
be a presentation of $M$. The zeroth Fitting ideal of $M$ over $R$, denoted $\mathrm{Fitt}_R(M)$, is the ideal of $R$ generated by the $r\times r$ minors of the matrix $h$. The Fitting ideal does not depend of the choice of presentation for $M$.

\begin{lemma}\label{lm:Fitt quotient}
Let $M$ be a finitely presented $R$-module. If $I \subset R$ is an ideal, then
$$
\mathrm{Fitt}_{R/I}(M/IM) = \pi \left( \mathrm{Fitt}_R(M) \right)
$$
where $\pi:R\to R/I$ is the natural quotient map.
\end{lemma}

\begin{proof}
    See \cite[4, Appendix]{MW84}.
\end{proof}

\begin{lemma}\label{lm:Fitting tensor}
    If $R \to R^\prime$ is a ring map, then $\mathrm{Fitt}_{R^\prime}(M\otimes_R R^\prime)$ is the ideal of $R^\prime$ generated by the image of $\mathrm{Fitt}_R(M)$.
\end{lemma}

\begin{proof}
    This follows from the fact that $\otimes_R R^\prime$ is right exact.
\end{proof}

\subsection{Definition}
\begin{definition}
Let $R$ be a Noetherian ring. A complex $M^{\bullet}$ of $R$-modules is perfect if there exists a quasi-isomorphism $P^\bullet \to M^\bullet$, where $P^\bullet$ is a bounded complex of projective $R$-modules of finite type. Denote by $D_\mathrm{parf}(R)$ the category of perfect $R$-modules. One says that $M^\bullet \in D_\mathrm{parf}(R)$ has perfect amplitude contained in $[a,b]$ (denoted $M^\bullet \in D_\mathrm{parf}^{[a,b]}(R)$) if the complex $P^\bullet$ above can be chosen in such a way that $P^i=0$ for every $i<a$ and $i>b$.
\end{definition}

\begin{definition}
Suppose that $M^\bullet \in D_\mathrm{parf}(R)$ is represented by a complex
$$
M^\bullet = \cdots \to M^0 \to M^1 \to M^2\to \cdots
$$
of finitely generated free $R$-modules such that $M^N=M^{-N}=0$ for all sufficiently large integers $N$. We define the Euler-Poincaré characteristic $\chi(M^\bullet)$ as 
$$
\chi(M^\bullet) \colonequals \sum_{n \in \ZZ}(-1)^n\mathrm{rank}_R(H^n(M)).
$$
\end{definition}
\subsection{Selmer complexes for elliptic curves}

Recall that $\Sigma$ is a finite set of places of $K$ containing the places at infinity and the ones dividing $N_Ep$. Let $\Sigma^\prime$ be the set of primes dividing $N_E$. For $v$ a prime in $K$, let $I_v \subseteq \Gal(\overline{K_v}/K_v)$ be the inertia group at $v$.
We denote by $\widetilde{\R\Gamma}_f(G_{K,\Sigma},A,\Delta_{\text{Gr}})$ the Selmer complex defined by using the Greenberg local conditions for $E$ \cite{Nek06} and we denote its cohomology by $\widetilde{H}_f^i(G_{K,\Sigma},A,\Delta_{\text{Gr}})$. We define $\widetilde{H}^i_{f,\mathrm{Iw}}(K_\infty/K,T,\Delta_{\text{Gr}})$ to be $\widetilde{H}_f^i(G_{K,\Sigma},T\otimes_{\Zp} \Lambda^\iota,\Delta_{\text{Gr}})$ where the local datum $\Delta_{\text{Gr}}$ is induced by the filtration $T^+ \otimes \Lambda^\iota \subset T \otimes \Lambda^\iota$. Then, Nekov\'{a}\v{r} duality theorem \cite[8.9.6.2]{Nek06} provides us with a canonical isomorphism
\begin{equation}\label{eqn:H2H1}
\widetilde{H}^2_{f,\mathrm{Iw}}(K_\infty/K,T,\Delta_{\text{Gr}})^\iota \cong \widetilde{H}_f^1(G_{K,\Sigma},A,\Delta_{\text{Gr}})^\vee.
\end{equation}

\begin{proposition}\label{Prop:Nekovar Fitting}
For every prime $v\in \Sigma^\prime$, suppose that $A^{G_v}=0$. Then, we have
$$
\mathrm{Fitt}_{\Lambda}(\widetilde{H}^2_{f,\mathrm{Iw}}(K_\infty/K,T,\Delta_{\text{Gr}})) =  \mathrm{char}_{\Lambda}(\widetilde{H}^2_{f,\mathrm{Iw}}(K_\infty/K,T,\Delta_{\text{Gr}})).
$$
\end{proposition}

\begin{proof}
By the definition of the Selmer complex $\widetilde{\R\Gamma}_f(G_{K,\Sigma},A,\Delta_{\text{Gr}})$, we have an isomorphism
$$
\widetilde{H}^0_f(G_{K,\Sigma},A,\Delta_{\mathrm{Gr}}) \xrightarrow{\sim} \ker \left( H^0(G_{K,\Sigma},A) \to \prod_{v\in\Sigma} H^0(G_{K_v},A) \right).
$$
This shows that $\widetilde{H}^0_f(G_{K,\Sigma},A,\Delta_{\mathrm{Gre}})=0$. Let $v$ be a prime of $K$ and let $\mathrm{Tam}_v(T,(\varpi))$ be the local Tamagawa factor as defined in \cite[Definition 7.6.10.1]{Nek06}. If $A^{G_v}=0$, then we have by \cite[7.6.10.12]{Nek06} that $\mathrm{Tam}_v(T,(\varpi))=0$. Thus, it follows by \cite[Proposition 9.7.7 (iii)]{Nek06} that 
\begin{equation}\label{eqn:perfect}
\widetilde{\R\Gamma}_{f,\mathrm{Iw}}(G_{K,\Sigma},T,\Delta_{\text{Gr}}) \in D_\mathrm{parf}^{[1,2]}(\Lambda).
\end{equation}
Furthermore, \cite[Theorem 8.9.15]{Nek06} gives us that the Euler-Poincaré characteristic of $\widetilde{\R\Gamma}_{f,\mathrm{Iw}}(G_{K,\Sigma},T,\Delta_{\text{Gr}})$ is 
$$
\mathrm{rank}_\Lambda (T \otimes \Lambda^\iota) - \sum_{v \in \{\gp,\gq\}}[K_v:\Qp]\mathrm{rank}_\Lambda (T_v^+\otimes \Lambda^\iota).
$$
By the definition of $T_v^+$ and the fact that $p$ splits in $K$, we get
\begin{equation}\label{eqn:EPC}
\chi(\widetilde{\R\Gamma}_{f,\mathrm{Iw}}(G_{K,\Sigma},T,\Delta_{\text{Gr}})) = 0.
\end{equation}
By combining \eqref{eqn:perfect} and \eqref{eqn:EPC}, there exists an exact sequence of $\Lambda$-modules
$$
0 \to \widetilde{H}^1_{f,\mathrm{Iw}}(K_\infty/K,T,\Delta_{\text{Gr}}) \to P \to P \to \widetilde{H}^2_{f,\mathrm{Iw}}(K_\infty/K,T,\Delta_{\text{Gr}}) \to 0,
$$
where $P$ is a free $\Lambda$-module of finite rank. The proof of \cite[Corollary 2.10]{BS21} shows that $\widetilde{H}^2_{f,\mathrm{Iw}}(K_\infty/K,T,\Delta_{\text{Gr}})$ is $\Lambda$-torsion if and only if $ \widetilde{H}^1_{f,\mathrm{Iw}}(K_\infty/K,T,\Delta_{\text{Gr}})=0$, in which case 
$$
\mathrm{Fitt}_{\Lambda}(\widetilde{H}^2_{f,\mathrm{Iw}}(K_\infty/K,T,\Delta_{\text{Gr}})) =  \mathrm{char}_{\Lambda}(\widetilde{H}^2_{f,\mathrm{Iw}}(K_\infty/K,T,\Delta_{\text{Gr}})).
$$
If $\widetilde{H}^2_{f,\mathrm{Iw}}(K_\infty/K,T,\Delta_{\text{Gr}})$ is not torsion, 
$$
\mathrm{Fitt}_{\Lambda}(\widetilde{H}^2_{f,\mathrm{Iw}}(K_\infty/K,T,\Delta_{\text{Gr}})) =0=  \mathrm{char}_{\Lambda}(\widetilde{H}^2_{f,\mathrm{Iw}}(K_\infty/K,T,\Delta_{\text{Gr}})).
$$
\end{proof}

\subsection{Relation to Greenberg's Selmer groups}

\begin{proposition}\label{prop:FittChar}
Suppose that we have $A^{G_v}=0$ for every prime $v \in \Sigma^\prime$. Then, we have 
$$
\mathrm{Fitt}_{\Lambda}(\mathcal{X}_{\text{Gr}}(K_\infty,A)) = \mathrm{char}_{\Lambda}(\mathcal{X}_{\text{Gr}}(K_\infty,A)).
$$
\end{proposition}

\begin{proof}
As remarked in \cite[Section 3.5]{BL18}, a slight generalization of \cite[Proposition 9.6.6(iii)]{Nek06} gives a canonical isomorphism
$$
\widetilde{H}_f^1(G_{K,S},A,\Delta_{\text{Gr}})^\vee \cong \mathcal{X}_{\text{Gr}}(E/K_\infty).
$$
If $M$ is any finitely generated torsion $\Lambda$-module, we have $\mathrm{Fitt}_{\Lambda}(M^\iota) = \iota\mathrm{Fitt}_{\Lambda}(M)$. Indeed, apply lemma \ref{lm:Fitting tensor} to the ring map $\iota:\Lambda\to \Lambda$. It also follows from the definition of the characteristic ideal that $\mathrm{char}_{\Lambda}(M^\iota) = \iota \mathrm{char}_{\Lambda}(M)$ (see for example the discussion of \cite[\S 15.5]{Wa97}. Thus, the result follows by proposition \ref{Prop:Nekovar Fitting} and the isomorphism \eqref{eqn:H2H1}.
\end{proof}

\section{Refined Mazur--Tate conjecture}

We now prove the main result of this note.

Let $\mathfrak{f} = \gp^{n+1} \gq^{m+1}$ be a modulus. Since taking $\Gal (K_\infty/K^\prime(\mathfrak{f}))$-invariants is the same as considering the part of $\Sel_\mathrm{Gr}(K_\infty,A)$  killed by $(\omega_n(\gamma_{\gp}-1),\omega_m(\gamma_{\gq}-1))$, the control theorem \eqref{prop:Control} can be reformulated as 
\begin{equation}\label{eqn:control2}
\Sel_\mathrm{Gr}(K_\infty,A)[\omega_n(\gamma_{\gp}-1),\omega_m(\gamma_{\gq}-1)] \cong \Sel_\mathrm{Gr}(K^\prime(\mathfrak{f}),A)
\end{equation}
provided that $E(K)[p]$ is trivial, that $a_p \not\equiv 1 \bmod p$ and that $E(K_v)[p]$ is trivial for all $v \in \Sigma^\prime$. Put $\Lambda_{n,m} \colonequals \Zp[\Gal( K^\prime(\gp^{n+1} \gq^{m+1})/K)]$.

\begin{theorem}\label{thm:main}
    Assume that the Iwasawa main conjecture \ref{conj:IMC} holds. Assume that the ratio of periods $\Omega_{\mathrm{Loe}}/\langle f,f\rangle_{N_E}$ is a $p$-adic unit. Suppose that $E(K)[p]$ is trivial, that $E(K_v)[p]$ is trivial for all $v \in \Sigma^\prime$ and that $a_p \not\equiv 1 \bmod p$. then,
    $$
    \mathrm{Fitt}_{\Lambda_{n,m}}(\mathcal{X}_{\mathrm{Gr}}(K^\prime(\gp^{n+1}\gq^{m+1}),A)) = (\Theta_{n,m})
    $$
    as ideals in $\Lambda_{n,m}$ for $n+m \geq 1$.
\end{theorem}

\begin{proof}
    By the Iwasawa main conjecture \ref{conj:IMC}, proposition \ref{prop:L-functions} and proposition \ref{prop:FittChar}, we have an equality of ideals
    \begin{equation}\label{eqn:1}
    \mathrm{Fitt}_\Lambda(\mathcal{X}_{\mathrm{Gr}}(K_\infty,A)) = \mathrm{Char}_\Lambda(\mathcal{X}_{\mathrm{Gr}}(K_\infty,A)) = (L_p^{\mathrm{Hi}}(f_\alpha/K)) = (L_{p,\Delta}(f))
    \end{equation}
    inside $\Lambda$. Take the projection to $\Lambda_{n,m}$ on both side of \eqref{eqn:1} and use lemma \ref{lm:Fitt quotient} to get
    $$
    (\Theta_{n,m}^S) = \mathrm{Fitt}_{\Lambda_{n,m}}(\mathcal{X}_{\mathrm{Gr}}(K_\infty,A)/(\omega_n(\gamma_\gp-1),\omega_m(\gamma_\gq-1))\mathcal{X}_{\mathrm{Gr}}(K_\infty,A)).
    $$
    Finally, since Mazur--Tate elements and $p$-stabilized Mazur--Tate elements generate the same ideal by theorem \ref{thm:stabilized}, we get
    \begin{align*}
    (\Theta_{n,m}) &= \mathrm{Fitt}_{\Lambda_{n,m}}(\Sel_\mathrm{Gr}(K_\infty,A)[\omega_n(\gamma_{\gp}-1),\omega_m(\gamma_{\gq}-1)]^\vee ) \\
    &= \mathrm{Fitt}_{\Lambda_{n,m}}(\Sel_\mathrm{Gr}(K^\prime(\gp^{n+1}\gq^{m+1}),A)^\vee) \\
    &= \mathrm{Fitt}_{\Lambda_{n,m}}(\mathcal{X}_{\mathrm{Gr}}(K^\prime(\gp^{n+1}\gq^{m+1}),A))
    \end{align*}
    where the second equality follows by Mazur's control theorem \eqref{eqn:control2}.
\end{proof}

By using the proof of the Iwasawa main conjecture \ref{thm:IMC} given by Castella and Wan in the indefinite setting, we obtain a formulation of the previous theorem that is less conjectural.

\begin{corollary}
    Assume that
    \begin{itemize}
        \item $N^-$ is the product of an even number of prime in $K$;
        \item $N_E$ is squarefree;
        \item $N^- \neq 1$;
        \item $\overline{\rho}_f$ is ramified at every prime dividing $N^-$;
        \item $\mathbb{I}$ is regular;
        \item $E(K)[p]$ is trivial;
        \item $E(K_v)[p]$ is trivial for all $v \in \Sigma^\prime$;
        \item $a_p \not\equiv 1 \bmod p$.
    \end{itemize}
    Then,
    $$
    \mathrm{Fitt}_{\Lambda_{n,m}\otimes \Qp}(\mathcal{X}_{\mathrm{Gr}}(K^\prime(\gp^{n+1}\gq^{m+1}),A)) = (\Theta_{n,m})
    $$
    as ideals in $\Lambda_{n,m}\otimes \Qp$ for $n+m \geq 1$.
\end{corollary}

\begin{proof}
    The proof is the same as theorem \ref{thm:main}, but now the equality of ideals
    \begin{equation*}
    \mathrm{Fitt}_\Lambda(\mathcal{X}_{\mathrm{Gr}}(K_\infty,A)) = \mathrm{Char}_\Lambda(\mathcal{X}_{\mathrm{Gr}}(K_\infty,A)) = (L_p^{\mathrm{Hi}}(f_\alpha/K)) = (L_{p,\Delta}(f))
    \end{equation*}
    takes place inside $\Lambda \otimes \Qp$.
\end{proof}

\begin{remark}
    In the definite case, Skinner and Urban \cite{SU14} proved a two-variable Iwasawa main conjecture. Their result shows that a multiple of the Hida $p$-adic $L$-function appearing in this note with some Euler factors removed is a generator of the characteristic ideal of the dual $\Sigma$-primitive Selmer group (see \cite[Section 3.1.3]{SU14} for the definition). To prove the strong Mazur--Tate conjecture in this setting, one would need to compare the Fitting ideal of the dual Selmer group with the one of the dual $\Sigma$-primitive Selmer group.
\end{remark}

\bibliographystyle{acm}
\bibliography{main.bib}
\end{document}